\documentclass[11pt]{amsart}
\usepackage[utf8]{inputenc}

\usepackage{amsfonts}
\usepackage{amsmath}
\usepackage{amsthm}
\usepackage{amssymb}
\usepackage{mathrsfs}
\usepackage{enumitem}
 
\usepackage[all]{xy}
\usepackage{appendix}

\theoremstyle{definition}
\newtheorem{defi}{Definition}[section]
\newtheorem{eg}[defi]{Example}
\newtheorem{rem}[defi]{Remark}
\newtheorem{prop}[defi]{Proposition}
\newtheorem{cor}[defi]{Corollary}
\newtheorem{lem}[defi]{Lemma}
\newtheorem{theo}[defi]{Theorem}

\newcommand{\A}{\mathsf{A}}

\newcommand{\C}{\mathsf{C}}

\newcommand{\dgvect}{\mathsf{DGVect}(\mathsf{k})}
\newcommand{\dgcat}{\mathsf{DGCat}(\mathsf{k})}
\newcommand{\dgalg}{\mathsf{DGAlg}(\mathsf{k})}

\newcommand{\cdgalg}{\mathsf{CDGAlg}(\mathsf{k})}

\newcommand{\Hom}{\operatorname{Hom}}

\newcommand{\holim}{\operatorname{holim}}
\newcommand{\id}{\operatorname{id}}

\newcommand{\co}{\colon\thinspace}

\title{Homotopy characters as a homotopy limit}

\author{Sergey Arkhipov}
\author{Daria Poliakova}

\begin{document}
\maketitle
 \begin{abstract}
  For a Hopf DG-algebra corresponding to a derived algebraic group, we compute the homotopy limit of the associated cosimplicial system of DG-algebras given by the classifying space construction. The homotopy limit is taken in the model category of DG-categories. The objects of the resulting DG-category are Maurer-Cartan elements of $\operatorname{Cobar}(A)$, or 1-dimensional $A_\infty$-comodules over $A$. These can be viewed as characters up to homotopy of the corresponding derived group. Their tensor product is interpreted in terms of Kadeishvili's multibraces. We also study the coderived category of DG-modules over this DG-category.
\end{abstract}
 \tableofcontents
\section{Introduction}
The note is devoted to an explicit calculation of a homotopy limit for a certain cosimplicial diagram in the model category of DG-categories. Recall that a general construction for  representatives of such  derived limits was given in the papers \cite{BHW} and \cite{AO2}. Below we consider a baby example where the answer appears to be both explicit and meaningful. \\

Let us illustrate our answer in an important special case. Take the Hopf algebra $A$ of regular functions on an affine algebraic group $G$. The cosimplicial system we consider is given basically by the simplicial scheme $X_\bullet$ realizing $BG$. Notice that if we considered the DG-categories of quasicoherent sheaves on $X_n$ and passed to the homotopy limit, the resulting DG-category would have been a model for the derived category of quasicoherent sheaves on the classifying space  $BG$ which is known to be equivalent to the derived category of representations of $G$. \\

Our task is different: we treat the (DG)-algebras of regular functions on $X_n$ as DG-categories with one object and consider the corresponding homotopy limit. We prove that it is equivalent to an interesting subcategory in the category of representations up to homotopy introduced earlier by Abad, Crainic, and Dherin (see \cite{ACD}): the (non-additive) DG-category of {\em characters up to homotopy} of the group $G$ also known as the DG-category of Maurer-Cartan elements in the Cobar construction for the coalgebra of functions on $G$.\\

The obtained answer illustrates a delicate issue: taking homotopy limit of a diagram of DG-categories does not commute with the (infinity-) functor $\A\mapsto \mathsf{DGMod}(\A)$.  Namely, passing to the categories of modules levelwise and then considering the homotopy limit would have produced the DG-category of quasicoherent sheaves on $BG$. Yet applying $\mathsf{DGMod}(\ldots)$ to the DG-category of homotopy characters we get a different category. \\

However,  if we replace the {\em derived categories} of DG-modules by the coderived ones,  this difference of the answers vanishes: the coderived category of DG-modules over the DG-category of homotopy characters for $G$ is quasi-equivalent to the coderived category of DG-modules over endomorphisms of the trivial character. By Positselski Koszul duality, the latter category is quasi-equivalent to the coderived category of representations for $G$.  \\

We conclude the paper by constructing an associative tensor product of objects in the DG-category of characters up to homotopy (in the generality of a DG-Hopf algebra, since we never use commutativity of the algebra in our considerations). Recall that Abad, Crainic and Dherin also constructed a homotopy monoidal structure on their category of representations up to homotopy (see \cite{ACD}). Our answer agrees with theirs. We interpret this answer in terms of Kadeishvili's multibraces. \\

Notice that there is no expectation to produce a honest associative tensor product of morphisms before passing to the homotopy category. Instead we plan to produce a homotopy coherent data descending to this structure after taking homology. This is work in progress. \\

\subsection*{Organization of the paper} In Section \ref{prelim} we give preliminaries on model categories, DG-modules and Cobar-constructions. In Section \ref{system} we introduce the cosimplicial system of interest, state its homotopy limit in the category of DG-algebras, and give the first description of its homotopy limit in the category of DG-categories. In Section \ref{mcsection} we interpret this result in terms of Maurer-Cartan elements in Cobar-construction. In Section \ref{comorita} we explain the coMorita equivalence between the homotopy limit taken in the category of DG-algebras and the homotopy limit taken in the category of DG-categories. In Section \ref{homchar} we reinterpret the homotopy limit category in terms of representations up to homotopy in the sense of \cite{AC}. In Section \ref{monoidal} we discuss the monoidal structure (as in \cite{ACD}) and how it is connected to Kadeishvili's multibraces. Finally in Appendix \ref{appen} we provide a detailed computation of the same homotopy limit in the category of DG-algebras, by means of simplicial resolutions. %(we believe that this is folklore knowlegde but we could not find a proof).
\subsection*{Acknowledgements} We are grateful to Leonid Positselski for many enlightening comments, in particular for sharing the proof of Lemma \ref{co} with us. The second author would like to thank Timothy Logvinenko for inviting her to present an early version of this project at GiC seminar in Cardiff, and Ryszard Nest for useful discussions. The  second  author  was  supported  by  the  Danish  National Research Foundation through the Centre for Symmetry and Deformation (DNRF92).

\section{Preliminaries}
\label{prelim}
\subsection{Model categories involved}
The category of DG-algebras $\dgalg$ is equipped with projective model structure which is right-transferred from the category of chain complexes along the adjunction ``tensor algebra functor/forgetful functor". The weak equivalences are the quasiisomorphisms, the fibrations are the surjections, and the cofibrations are defined by the left lifting property. \\

%The same holds verbatim for the category of {\it commutative} DG-algebras $\comdgalg$, where the adjunction with the category of chain complexes is "symmetric algebra functor/forgetful functor". \\

In this paper, we mostly work with more general objects. Recall that a DG-category is by definition a category enriched over the monoidal category of complexes of vector spaces, denoted by $\dgvect$. Every DG-algebra is a DG-category with one object. We denote the category of small DG-categories and DG-functors over a field $\mathsf{k}$ by $\dgcat$. Tabuada constructed a model category structure on $\dgcat$, with weak equivalences being quasi-equivalences of DG-categories (see \cite{Tab}).\\

For an arbitrary model category $\mathsf{C}$, the category $\mathsf{C}^{\Delta^{opp}}$ is equipped with Reedy model structure (see \cite{Hov} or \cite{Hir}).

\subsection{DG-modules} A DG-functor from a DG-category $\A$ with values in the DG-category $\dgvect$ is called an $\A$-DG module. Notice that this agrees with the definition of a DG-module over a DG-algebra. The DG-category of $\A$-DG-modules is denoted by $\mathsf{DGMod}(\A)$.\\

\subsection{Cobar-constructions} In our paper we will be dealing with two sorts of Cobar-construction for DG-coalgebras. In the first construction, the complex happens to be acyclic whenever the coalgebra is counital; conceptually it is a cofree resolution of the coalgebra as a comodule over itself. In the second construction, the coaugmentation of the coalgebra provides boundary terms for the differential; the resulting complex is quasiisomorphic to what is known as {\emph reduced} Cobar-construction, and it is similar to the standard complex computing $\operatorname{Cotor}^C(\mathsf{k},\mathsf{k})$. Note however, that in this note we are using products not sums. Let us give the definitions and the notation.

\begin{defi}
Let $C$ be (not necessarily counital or coaugmented) DG-coalgebra. As a graded  vector space, 
$$\operatorname{Cobar}(C) = \widehat{T}(C[-1]) = \prod_{i=0}^\infty C[-1]^{\otimes i}$$
The multiplication is that of a complete tensor algebra. The differential is given by $d = d_C + \Delta $ on generators and extends to the rest of the algebra by Leinbiz rule.
\end{defi}

\begin{rem}
If $C$ is counital, this Cobar construction is actually acyclic, with counit giving rise to a contraction.
\end{rem}

If C is coaugemented with coaugmentation $1\co \mathsf{k} \to C$, then there is the following modification.

\begin{defi}
As a graded algebra, $\operatorname{Cobar}_{\operatorname{coaug}}(C) \simeq \widehat{T}(C[-1])$ again. The differential is given by $d = d_C + \Delta + 1 \otimes \id - \id \otimes 1$ on generators and extends to the rest of the algebra by Leinbiz rule.
\end{defi}

\begin{rem}
In the coaugmented case, $1_{C}$ is a Maurer-Cartan element in $\operatorname{Cobar}(C)$, and the differential in the later construction is the differential in the former construction twisted by this Maurer-Cartan element.
\end{rem}

\section{The cosimplicial system}
\label{system}
Let $(A, m, 1, \Delta, \epsilon)$ be a (unital, counital)  DG-bialgebra. Informally, in the case when $A$ is commutative we should view it as the algebra of functions on a derived affine algebraic group scheme. Notice however that we never use commutativity of $A$ in our main statements. \\

Consider the cosimplicial system $A^\bullet$ of DG-algebras corresponding to the classifying space construction:
\begin{equation}
\label{eq:cosys}
\xymatrix {k \ar@<-.5ex>[r] \ar@<.5ex>[r] & A \ar@<-1ex>[r] \ar[r] \ar@<1ex>[r] & A^{\otimes 2} & \cdots } 
\end{equation}

Let $\partial^i_n$ denote the face map $A^{\otimes n} \to A^{\otimes n+1}$ and $s^i_n$ denote the degeneracy map $A^{\otimes n} \to A^{\otimes n-1}$. Then in the system above
with faces and degeneracies given by
$$\partial^i_n = \begin{cases} 1 \otimes \id^{\otimes n} & i = 0 \\
\id^{\otimes i-1} \otimes \Delta \otimes \id^{\otimes n-i} & 0 < i < n+1 \\
\id^{\otimes n} \otimes 1 & i = n+1 \end{cases}$$

$$s^i_n = \id^{\otimes i} \otimes \epsilon \otimes \id^{\otimes n-i-1}$$

There are several homotopy limit computations that can be done in relation to system \eqref{eq:cosys}:
\begin{enumerate}[label=(\alph*)]
    \item One can compute the homotopy limit in the category of DG-algebras
    \item One can view every DG-algebra as a DG-category with one object and compute the homotopy limit in the category of DG-categories
    \item One can apply DG-Mod functor and compute the homotopy limit of this new system of DG-categories.
\end{enumerate}

The answer to (a) is folklore. The homotopy limit of the cosimplicial system is given by reduced Cobar-construction of the corresponding coaugmented DG-coalgebra. We were not able to locate the proof of this statement in the literature, thus we reproduce it in Appendix \ref{appen}.\\

In this paper we mainly discuss the answer to (b). The comparison between (b) and (c) is discussed in Section \ref{comorita}.\\
%Rather that dealing with  the cosimplicial  DG-algebra (\ref{eq:cosys}), we view the system above as a cosimplicial DG-category, i.e. an object of $\mathsf{DGCat}(k)^{\Delta^{\operatorname{op}}}$. It is known that the category $\mathsf{DGCat}(k)^{\Delta^{\operatorname{op}}}$ has several model structures, each having levelwise quasiequivalences as weak equivalences,  and we want to obtain an explicit description  for the derived limit of our cosimplicial DG-category.\\

In the papers \cite{BHW} and \cite{AO2} the authors realized homotopy limits in $\dgcat^{\Delta^{\operatorname{op}}}$ as derived totalizations. Below we cite Prop. 4.0.2 from \cite{AO2}, with formulas written in their most explicit form. To simplify the notation, we denote by $\partial^{(i_1\ldots i_k)}$ an inclusion with image $i_1, \ldots, i_k$.
\begin{theo}
\label{ao}
For $\C^\bullet$ a cosimplicial system of DG categories, an object of $\holim \C$ is the data of $(X,a=\{a_i\}_{i \geq 1})$, where $X$ is an object of $\C^0$ and $a_i \in \Hom_{\C^i}^{1-i}(d^{(0)}X,d^{(n)}X)$ with $a_1$ homotopy invertible and subject to
\begin{equation}
\label{eq:object}
\begin{split}
   & d(a_n) = -\sum_{k=1}^{n-1}(-1)^{n-k} \partial^{(k \ldots n)} (a_{n-k}) \circ \partial^{(0 \ldots k)}(a_k) \\ & +\sum_{k=1}^{n-1}(-1)^{n-k} \partial^{(0 \ldots \hat{k} \ldots n )}(a_{n-1}).
\end{split}
\end{equation} 
The complex of morphisms between $(X,a)$ and $(Y,b)$ in degree $m$ is given by
$$\Hom^{m}((X,a),(Y,b)) = \prod_{i = 0}^\infty \Hom_{\C^i}^{m-i}(\partial^{(0)}(X),\partial^{(i)}(Y))$$
where we read $\partial^{(0)} \co \C^0 \to \C^0$ as $\id_{\C^0}$. For $f = \{f_i\} \in \Hom^{m}((X,a),(Y,b))$ its differential is given by

\begin{equation}
\label{eq:aodiff}
\begin{split}
    & d(f)_n = d(f_n)+\sum (-1)^{n-k} \partial^{(k \ldots n)} (f_{n-k}) \circ \partial^{(0\ldots k)}(a_{k}) \\
& - \sum_{k=1}^{n-1} (-1)^{m(n-k+1)}\partial^{(k \ldots n)} (b_{n-k}) \circ \partial^{(0\ldots k)}(f_{k})  \\
& +\sum_{k=1}^{n-1}(-1)^{n-k+m} \partial^{(0 \ldots \hat{k} \ldots n )}(f_{n-1}).
\end{split}
\end{equation}

For $f \in \Hom^{m}((X,a),(Y,b))$ and $g \in \Hom^{l}((Y,b),(Z,c))$, their composition composition is given by 
\begin{equation}
\label{eq:aomul}
    (g \circ f)_n =  \sum_{k=0}^{n} (-1)^{m(n-k)} \partial^{(k \ldots n)}(g_{n-i}) \circ \partial^{(0\ldots k)}(f_k). 
\end{equation}
\qed
\end{theo}

We now apply these formulas to the cosimplicial system \eqref{eq:cosys}. Note that while each category in \eqref{eq:cosys} has a single object, this would not hold for the homotopy limit, where the data of an object includes morphisms. Denote $\holim A^\bullet = : \mathfrak{A}$. 

\begin{theo}
An object $a$ in $\mathfrak{A}$ is an infinite sequence $\{a_i \}_{i \geq 1}$ with $ a_i \in (A^{\otimes i})^{1-i} $ and $a_1$ homotopy invertible, subject to relations
\begin{equation}
\label{eq:form}
\begin{aligned}
& d(a_1)=0 \\
& d(a_2) = a_1 \otimes a_1 - \Delta(a_1) \\
& \ldots \\
& d(a_n) = - \sum_{k=1}^{n-1} (-1)^{n-k} a_{n-k} \otimes a_{k} \\ & + \sum_{k=1}^{n-1} (-1)^{n-k}(\id^{\otimes k-1}\otimes \Delta \otimes \id^{\otimes n-k-1})(a_{n-1})\\
& \ldots
\end{aligned} 
\end{equation}
A morphism $f\co a \to b$ of degree $m$ is also an infinite sequence $\{f_n \}_{n \geq 0}$ with $f_n \in (A^{\otimes n})^{-n}$, with differential given by 
\begin{equation}
\label{eq:diff}
\begin{split}
& d(f)_n = d(f_n) + \sum_{k=1}^{n-1}(-1)^{n-k} a_k \otimes f_{n-k}  - \sum_{k=1}^{n-1} (-1)^{m(n-k+1)} f_i \otimes b_{n-k} \\
& + \sum_{k=1}^{n-1} (-1)^{n-k+m}(\id^{\otimes k-1}\otimes \Delta \otimes \id^{\otimes n-k-1})(f_{n-1})
\end{split}
\end{equation}
and composition given by
\begin{equation}
\label{eq:mult}
 (g \circ f)_n = \sum_{k=0}^n (-1)^{m(n-k)} g_n \otimes f_{n-k}   
\end{equation}

\end{theo}

\begin{proof}
This is a straightforward application of Theorem \ref{ao}. As in the theorem, denote an object of the homotopy limit by $(X,a)$. In our cosimplicial system \eqref{eq:cosys}, $A^0 = \mathsf{k}$ has only one object, so $X = *$. Then the identities \eqref{eq:object} translate to \eqref{eq:form}, the formula for the differential \eqref{eq:aodiff} corresponds to \eqref{eq:diff}, and the formula for the composition \label{eq:aomult} corresponds to \eqref{eq:mult}.
\end{proof}

Below we present several interpretations of this data.

\section{Maurer-Cartan elements in Cobar}
\label{mcsection}

We interpret the homotopy limit category $\mathfrak{A}$ in terms of Cobar construction for the DG-coalgebra $A$.

\begin{prop}
The objects of $\mathfrak{A}$ are exactly the Maurer-Cartan elements of $\operatorname{Cobar}(A)$, with one extra condition that their first component is homotopy invertible.
\end{prop}
\begin{proof}
The Maurer-Cartan  equation $dx+\frac{1}{2}[x,x]=0$ translates precisely into the formulas \eqref{eq:form}.
\end{proof} 

In any DG algebra $A$ a Maurer-Cartan element $c$ allows to twist the differential:
$$d_c(a) = d(a)+[c,a]$$
Denote the new algebra by $_cC_c$. For two Maurer-Cartan elements $c_1$ and $c_2$, denote by $_{c_1}C_{c_2}$ a complex obtained by considering $A$ with the new differential
\begin{equation}
\label{eq:twist}
d_{(c_1,c_2)}(a) = d(a) + c_1a - (-1)^{|a|}ac_2.
\end{equation}
This will not be a DG-algebra anymore (for the lack of multiplication satisfying the Leibniz rule), but it will be a $_{c_1}C_{c_1}$-$_{c_2}C_{c_2}$ DG-bimodule.

\begin{prop}
\label{hom}
In the DG-category $\mathfrak{A}$, the complex of morphisms
$$\mathfrak{A}(a,b) = {}_a\operatorname{Cobar}(A) _b.$$
\end{prop}
\begin{proof}
The formula \eqref{eq:twist} for the twisted differential corresponds precisely to the formula \eqref{eq:diff}.
\end{proof}

So as a graded vector space, every $\mathfrak{A}(a,b)$ is always equal to $\operatorname{Cobar}(A)$.

\begin{prop}
Under this assignment, the composition $\mathfrak{A}(a,b) \otimes \mathfrak{A}(b,c) \to \mathfrak{A}(a,c)$ corresponds to the multiplication in $\operatorname{Cobar}(A)$.
\end{prop}

\begin{proof}
This is the formula \eqref{eq:mult}.
\end{proof}

In $\operatorname{Cobar}(A)$, there is a distinguished nontrivial Maurer-Cartan element, namely, $1_A \in A$. Denote the corresponding object of $\mathfrak{A}$ by $\mathbb{I}$. Its endomorphisms are ${_{1_A}} \operatorname{Cobar}(A)_{1_A} \simeq \operatorname{Cobar}_{\operatorname{coaug}}(A)$. As explain in Appendix \ref{appen}, this is a model for the homotopy limit of our cosimplicial system but taken in the category $\dgalg$. \\

Recall the notion of gauge equivalence for Maurer-Cartan elements.

\begin{defi}
In a DG-algebra $A$, the gauge action of a degree $0$ invertible element $f$ on a Maurer-Cartan element $a$ is given by 
$$f.a = faf^{-1}+fd(f^{-1}).$$
\end{defi}

One checks that this is again a Maurer-Cartan element. Two Maurer-Cartan elements are called gauge equivalent if they belong to the same orbit of gauge action.

\begin{prop}
Gauge equivalent Maurer-Cartan elements of $\operatorname{Cobar}(A)$ are strictly isomorphic as objects of $\mathfrak{A}$.
\end{prop}
\begin{proof}
The very same invertible element provides the closed isomorphism when viewed as an element of the $\Hom$-complex. Upon explicitly checking closedness, the rest follows from composition being reinterpreted as the multiplication in $\operatorname{Cobar}(A)$.
\end{proof}

\section{CoMorita equivalences}
\label{comorita}
For any DG algebra $A$ and Maurer-Cartan elements $a$, $b$ it holds that
$$ {}_a A _b \otimes_{ {}_b A _b} {}_bA_a = {}_a A _a,$$
so on the nose $ _a A _b$ and $  _bA_a$ are inverse bimodules. This gives an expectation for a Morita equivalence between $\mathfrak{A}$ and $\operatorname{Cobar}(A)$. However, sometimes these bimodules may be acyclic, and derived tensoring by an acyclic bimodule cannot induce an equivalence of derived categories. To make things work one needs to consider not derived categories but instead Positselski's coderived categories, where the class of acyclic objects is replaced by a smaller class of coacyclic objects. For detailed exposition see \cite{P}.

\begin{defi}
\label{coacycl}
For a DG algebra $A$, the subcategory $\mathsf{CoAcycl} \subset \mathsf{Ho}(A)$ is the smallest triangulated subcategory containing totalizations of exact triples of modules and closed with respect to infinite direct sums.
\end{defi}

\begin{defi}
\label{dco}
The coderived category $\mathsf{D}^{co}(A)$ is defined as the Verdier quotient of the homotopy category  $\mathsf{Ho}(A)$ by the full subcategory $\mathsf{CoAcycl}$.
\end{defi}

For the proof of the next lemma, recall the notion of CDG-algebras and their morphisms. 

\begin{defi}
A curved DG-algebra (for brevity, a CDG-algebra) is a graded algebra $A$  equipped with a degree 1 derivation $d$ and a closed curvature element $h \in A^2$, satisfying 
$$d^2(x) = [h,x]$$
A morphism of CDG-algebras $A \to B$ is a pair $(f,b)$ where $f: B \to C$ is a multiplicative map and $c \in B^1$ is the change of curvature, i.e. they satisfy
\begin{equation}
\label{curve}
    f(d_A(x)) = d_B(f(x))+[a,x]
\end{equation}
\begin{equation}
\label{curve2}
d(h_A) = h_B + d_B(b)+b^2
\end{equation}

The composition of CDG-morphisms is 
$$(g,c) \circ (f,b) = (g \circ f, c + g(b))$$
\end{defi}

A DG-algebra can be viewed as a CDG-algebra with zero curvature, but the inclusion $\dgalg \hookrightarrow \cdgalg$ is not full.

\begin{lem}
\label{co}
For any DG algebra $A$ there is an equivalence of coderived categories 
$$\mathsf{D}^{co}(_aA_a) \simeq \mathsf{D}^{co}(_bA_b)$$
\end{lem}

\begin{proof}
$_aA_a$ and $_bA_b$ are isomorphic as CDG-algebras (with zero curvature). The CDG-isomorphism $_aA_a \to _bA_b$ is given by $(id,-a)$, where \eqref{curve} corresponds to the formula for twisting the differential, and \eqref{curve2} corresponds to Maurer-Cartan equation for $a$. Coderived categories are preserved under CDG-isomorphisms.
\end{proof}
\begin{rem}
Compare the calculation above of the explicit representative for the homotopy limit of the DG-algebras considered as DG-categories with the following. 
\begin{enumerate}
\item
In the paper \cite{AO2} the authors solve a similar problem for the homotopy limit {\em of the  derived categories} of DG-modules over the DG-algebras in the cosimplicial system. The answer can be interpreted as the derived category of DG-modules over the reduced Cobar construction for the original DG-Hopf algebra (Theorem 4.1.1).
\item
Conjecturally the statement remains true also for the homotopy limit of the corresponding enhanced  {\em coderived} categories: one obtains the coderived category of DG-modules over the Cobar construction for the original DG-Hopf algebra.
\end{enumerate}
Now take the category of DG-modules over the DG-category of Maurer-Cartan elements $\mathfrak{A}$. While its derived category obviously differs from the derived category that appears in (1), its {\em coderived} category is quasi-equivalent to the answer in (2).
%this follows from the described coMorita equivalences between all the categories $D^{co}(_aA_a)$. 
\end{rem}

We will now make this precise. Let $B$ be an arbitrary DG-algebra. 

\begin{defi}
Maurer-Cartan DG-category $\mathsf{MC}(B)$ has Maurer-Cartan elements of $B$ as morphisms, and Hom-complexes are given by $$\Hom_{\mathsf{MC}(B)}(a,b) = {}_a B_b.$$
\end{defi}

The definitions \ref{coacycl} and \ref{dco} can be directly generalized from DG-algebras to DG-categories, so for a DG-category $\mathcal{C}$ one can consider a category $D^{co}(\mathcal{C})$. 

\begin{prop}
\label{comor}
For any DG-algebra $B$ and a Maurer-Cartan element $b \in B$ there is an equivalence of categories 
$$\mathsf{D}^{co}(\mathsf{MC}(B)) \simeq \mathsf{D}^{co}(_b B _b).$$
\end{prop}

\begin{proof}
This is a statement of the type ``modules over a  connected groupoid are the same as modules over endomorphisms of an object in this groupoid", with a similar proof. \\

Let 
$$F\co \mathsf{DGMod}(\mathsf{MC}(B)) \to \mathsf{DGMod}(_b B_b )$$
be given by restricting to $b$,
$$F(M)=M(b).$$ Define
$$G \co \mathsf{DGMod}(_b B _b) \to \mathsf{DGMod}(\mathsf{MC}(B)) $$
by setting, for $a \in \mathsf{MC}(b)$, 
$$G(N)(a) = {}_aB_b\otimes_{_b B _b} N$$
and for $f \in \mathsf{MC}(B)(a_1,a_2) = {}_{a_1} B _{a_2}$ let the corresponding map 
$$G(f)\co _aB_b\otimes_{_bB _b} N \to {} _aB_b\otimes_{_b B _b} N$$
be simply multiplication by $f$ on the left. We would like to check that these functors induce an equivalence on coderived categories. First we check that they give an equivalence at the level of DG-categories. It is clear that $FG = Id_{\mathsf{DGMod}(_b B_b))}$. For $M \in \mathsf{DGMod}(\mathsf{MC}(B))$ and $a \in \mathsf{MC}(B)$, we have
$$GF(M)(a) = _a B _b \otimes_{_b B_b} M(b).$$
Then the isomorphism $GF(M) \to M$ is given at $a$ by 
$$ f \otimes m \mapsto M(f)(m)$$
and its inverse is 
$$ m \mapsto 1 \otimes M(1)(m)$$
where $1 \in {}_aB_b$ is viewed as a map $a \to b$.\\

We are left to verify that $F$ and $G$ preserve coacyclic objects. To do so, they need to preserve exact triples, and commute with totalizations, cones and infinite direct sums. For DG-modules over a DG-category, exactness is checked objectwise, and totalizations, cones and direct sums are also formed objectwise. Thus for $F$ the statements hold trivially. For $G$, the statements about totalizations, cones and sums hold trivially, and the statement that $G$ respects exact triples follows from flatness of $_b B_b$-modules $_a B _b$. They are indeed flat, because their underlying graded modules are just free of rank 1, and flatness does not depend on the differential.
\end{proof}

Note that in particular this proposition establishes a coMorita equivalence between $\mathsf{MC}(B)$ and $B$ itself, as $B$ can be seen as endomorphism algebra of $0 \in \mathsf{MC}(B)$. Also note that Lemma \ref{co} follows from this proposition, but we keep its proof via CDG-isomorphism because it is conceptually correct.

\begin{cor} There is an equivalence of coderived categories
$$\mathsf{D}^{co}(\mathfrak{A}) \simeq \mathsf{D}^{co}(\operatorname{Cobar}_{\operatorname{coaug}}(A)).$$ 
\qed
\end{cor}

Here we are considering reduced Cobar construction for the sake of comparing with the result in \cite{AO2} and with the computation in $\dgalg$. Reduced and non-reduced Cobar constructions are coMorita equivalent by Proposition \ref{comor}  (though not Morita equivalent).

\section{Homotopy characters}
\label{homchar}
Recall the notion of an $A_\infty$-comodule  over a DG-coalgebra ($A_\infty$-comodules can be considered over any $A_\infty$-coalgebra, but this generality will not be needed). For detailed exposition see \cite{AO2} or, on the dual side, \cite{Kel}. 

\begin{defi}
The $A_\infty$-comodule structure on a graded vector space $M$ over a DG-coalgebra $C$ is a DG-module structure on $M \otimes \operatorname{Cobar} (C) $ over $\operatorname{Cobar} (C)$. Explicitly, it is given by a sequence of coaction maps, for all $n \geq 1$,
$$ \mu_n \co M \to C^{\otimes n-1}\otimes M$$
with $\mu_n$ of degree $1-n$ and all the collection of maps together satisfying the $A_\infty$-identities for each $n \geq 1$:
\begin{equation}
\label{eq:ainfin}
\begin{split}
& (-1)^{n-1} \sum_{i=0}^n  (\id^{\otimes i} \otimes d \otimes \id ^{\otimes n-i-1}) \mu_{n} + \mu_{n} d  \\
& + \sum_{i=1}^{n-1} (-1)^{i}(\id^{\otimes i} \otimes \mu_{n-i})\mu_{i} + \sum_{i=0}^{n-2} (-1)^{i} (\id^{\otimes i } \otimes \Delta \otimes \id^{\otimes n-i-2} ) \mu_{n-1} = 0
\end{split}
\end{equation} 
\end{defi}

\begin{defi}
For two $A_\infty$-comodules over a DG-algebra $A$, $\Hom$-complex between them is defined by
$$\Hom^m(M,N) = \prod_{i=0}^\infty \Hom_k^{m-i}(M, C^{\otimes i} \otimes N)$$
with differential 
\begin{equation}
\begin{split}
  \label{eq:ainfdif}
& d(f)_n = \sum_{k=1}^{n-2} (-1)^{n-k}(\id^{\otimes n-k-2}\otimes\Delta\otimes\id^{\otimes k})f_{n-1} \\
& + \sum_{i=0} (-1)^{i} (\id^{\otimes i}\otimes \mu_{n-i})f_{i+1} + \sum_{p=1}^n (-1)^{p|f|}(\id^{\otimes p-1}\otimes f_{n-p+1})\mu_p
\end{split}
\end{equation}

The composition is given by
\begin{equation}
\label{eq:ainfcomp}
(g \circ f)_n = \sum_{l=1}^n (-1)^{|g|(l-1)} (\id ^{\otimes l-1} \otimes g_{n-l+1})f_l
\end{equation}

\end{defi}

\begin{prop}
The DG-category $\mathfrak{A}$ is isomorphic to the subcategory of 1-dimensional (non-counital) $A_\infty$-comodules over $A$. 
\end{prop}
\begin{proof}
For $M = \mathsf{k}$ a structure map $\mu_n \co \mathsf{k} \to A^{\otimes n} \otimes \mathsf{k}$ is indeed given by an element $a_n \in A^{\otimes n}$. The $A_\infty$-relations \eqref{eq:ainfin} correspond to the formulas \eqref{eq:form}. The formula for the differential \eqref{eq:ainfdif} corresponds to \eqref{eq:diff}, and the formula for the composition \eqref{eq:ainfcomp} corresponds to \eqref{eq:mult}.
\end{proof}

Note that if $A$ was the coalgebra of functions on some group, then comodules over this coalgebra would correspond to representations of the group. This leads us to the following interpretation of our data. $A_\infty$-comodules over a Hopf DG-algebra can be viewed as {\it representations up to homotopy} of the corresponding derived group. Within this category, one-dimensional comodules correspond to {\it homotopy characters}. Group representations up to homotopy have been defined and studied (for non-derived Lie groupoids) by Abad-Crainic in \cite{AC}. \\

In the case when $A$ is a Hopf algebra of functions on a group (concentrated in degree 0), our category  has honest characters as objects, and the Hom complexes  compute Exts between them.

\begin{eg}
Let $G$ be the group of invertible upper triangular $2 \times 2$ matrices over $\mathbb{C}$. Consider the following functions:

$$x \begin{pmatrix}
a & c \\
0 & b
\end{pmatrix} = a;\mbox{  }\mbox{  }
y \begin{pmatrix}
a & c \\
0 & b
\end{pmatrix} = b; \mbox{  }\mbox{  }
z \begin{pmatrix}
a & c \\
0 & b
\end{pmatrix} = c. 
$$

The Hopf algebra of regular functions on $G$ is $\mathbb{C}[x^{\pm1},y^{\pm1},z]$, with comultiplication
$$\Delta(x^{\pm1}) = x^{\pm1}\otimes x^{\pm1};$$
$$\Delta(y^{\pm1}) = y^{\pm1}\otimes y^{\pm1};$$
$$\Delta(z) = x \otimes z + z \otimes y.$$

$1$ and $xy^{-1}$ are two characters of $G$. We have $Ext^1(1,xy^{-1}) = \mathbb{C}.$ In our Holim category, the Hom complex between $1$ and $xy^{-1}$ is 

$$ \mathbb{C} \longrightarrow \mathbb{C}[x^{\pm1},y^{\pm1},z]\longrightarrow \mathbb{C}[x^{\pm1},y^{\pm1},z]^{\otimes 2} \longrightarrow \ldots$$

where the first differential is multiplication by $1-xy^{-1}$, and the second differential is given by $d(f) = f \otimes 1 + xy^{-1} \otimes f + \Delta(f)$. The kernel of it is generated by $1-xy^{-1}$ and $y^{-1}z$, the latter being a representative for the nontrivial first Ext.
\end{eg}

\section{Tensor products and multibraces}
\label{monoidal}
One can see that the data of multiplication in $A$ does not come up in the answer so far. This however suggests that $\mathfrak{A}$ is equipped with additional structure. We notice that a {\it commutative} DG-algebra is a {\it monoidal} DG-category with one object, and while the passage to homotopy limit might not preserve this structure, at least something can be expected to survive. Indeed, in \cite{ACD} the authors construct the monoidal structure on the homotopy category of all representations up to homotopy, which in particular restricts to the subcategory of characters. We obtain a similar answer for noncommutative DG-Hopf algebras as well.\\

Let $a = \{a_i \}$ and $b = \{b_i \}$ be two homotopy characters.  Then $a_1$ and $b_1$ are homotopy invertible and homotopy grouplike, and so is $a_1b_1$. Indeed, if  $a_1 \otimes a_1 - \Delta(a_1) = d(a_2)$ and $b_1 \otimes b_1 - \Delta(b_1) = d(b_2)$, then

\begin{align*}
& a_1b_1 \otimes a_1b_1 - \Delta(a_1b_1)   \\
& = (a_1 \otimes a_1)(b_1 \otimes b_1) - \Delta(a_1)\Delta(b_1) \\
& =  (\Delta(a_1)+d(a_2))(\Delta(b_1)+d(b_2))- \Delta(a_1)\Delta(b_1) \\
& = (\Delta(a_1)+d(a_2))d(b_2)+ d(a_2) \Delta(b_1)  \\
& = (a_1 \otimes a_1) d(b_2)+ d(a_2) \Delta(b_1)  \\
& = d((a_1 \otimes a_1)b_2+ a_2 \Delta(b_1)).
%& = d(a_2)(\Delta(b_1)+d(b_2)) + \Delta(a_1)d(b_2)= \\
%& = d(a_2)(b_1 \otimes b_1) + \Delta (a_1) d(b_2) = \\
%& = d( a_2 (b_1 \otimes b_1) + \Delta(a_1) b_2).
\end{align*}

We notice that  $(a_1b_1, (a_1 \otimes a_1) b_2 + a_2\Delta(b_1), \ldots)$ starts looking like the beginning of another homotopy character. There is an asymmetry between $a$ and $b$, but there is a certain freedom to modify the formulas above, so we could have also obtained $(a_1b_1, a_2 (b_1 \otimes b_1) + \Delta(a_1) b_2, \ldots)$.

\begin{theo}
Let $a = (a_1, a_2, \ldots)$ and $b = (b_1,b_2, \ldots)$ be homotopy characters. Then there exists a homotopy character $a \otimes b$, given by the formulas 

\begin{equation}
\label{eqref:tensor}
(a \otimes b)_n = \sum_{i_1 + \ldots + i_k = n} (a_{i_1} \otimes \ldots \otimes a_{i_k})(\Delta^{i_1-1}\otimes \ldots \otimes \Delta^{i_k-1})(b_n).
\end{equation} 

There also exists a homotopy character given by the formulas

\begin{equation}
\label{eqref:tensor2}
(a \otimes b)_n = \sum_{i_1 + \ldots + i_k = n} (\Delta^{i_1-1}\otimes \ldots \otimes \Delta^{i_k-1})(a_n) (b_{i_1} \otimes \ldots \otimes b_{i_k})
\end{equation} 

Both tensor products of objects are strictly associative.
\end{theo}
\begin{proof}
It can be explicitly checked that Maurer-Cartan equation holds in both cases.
%$$d(a \otimes b)_n = \sum_{i=1}^{n-1} (\id^{\otimes i-1} \otimes \Delta \otimes 1^{\otimes n-1-i}) (a \otimes b)_{n-1} + \sum (a \otimes b)_{i} \otimes (a \otimes b)_{n-i}$$
Strict associativity of these tensor products is obtained by a direct computation. 
\end{proof}

The formulas above are the same as in Corollary 5.10 in \cite{ACD} -- in their notation, these are $\omega_0$ and $\omega_1$. Theorem 5.6 in \cite{ACD} states that the two different tensor products are actually homotopy equivalent.\\

%The tensor products above are not strictly-symmetric, but for both of them there is a homotopy equivalence between $a \otimes b$ abd $b \otimes a$, given by: (write a formula!)\\

The formulas \eqref{eqref:tensor} and \eqref{eqref:tensor2} have an interpretation in terms of Kadeishvili's multibraces that exist on the $\operatorname{Cobar}$-construction of a bialgebra and assemble into homotopy Gerstenhaber algebra structure. Recall the following definitions.

\begin{defi}
For a DG-algebra $B$ with multiplication $\mu$, its $\operatorname{Bar}$-construction is, as a graded vector space, 
$$\operatorname{Bar}(B) = {T}(B[1]) = \bigoplus_{i=0}^\infty B[1]^{\otimes i}.$$
The comultiplication is that of a tensor coalgebra. The differential is given by $d = d_B + \mu$ into the cogenerators and extends to the rest of the coalgebra by coLeinbiz rule.
\end{defi}

\begin{defi}
\label{hga}
A DG-algebra $B$ is a homotopy Gerstenhaber algebra (hGa) if it is equipped with a family of operations (multibraces) $$E_{1_,k}\co B \otimes B^{\otimes k} \to B$$ that induce a associative multiplication on $\operatorname{Bar}(B)$ consistent with its tensor comultiplication.
\end{defi}

\begin{rem}
A multiplication on $\operatorname{Bar}(B)$ is a coalgebra map $$E\co \operatorname{Bar}(B)\otimes  \operatorname{Bar}(B) \to \operatorname{Bar}(B).$$ As a coalgebra map, it is uniquely determined by its part that lands into the cogenerators, $B$. Denote its component $B^{\otimes l} \otimes B^{\otimes k} \to B$ by $E_{l,k}$. A family of $E_{l,k}$ that gives rise to an associative multiplication is known as Hirsch algebra structure on $B$. In Definition \ref{hga} we restrict ourselves to families where $E_{l,k}$ vanish when $l \neq 1$.
\end{rem}

For elements $b$ and $b_1$, $\ldots$, $b_k$ we write $E_{1,k}(b;b_1, \ldots, b_k) = b\{b_1, \ldots b_k \}$ (thus the term multibraces). We can naturally modify the definitions above to also obtain operations $E_{k,1}$, for which we will write $E_{1,k}(b_1, \ldots, b_k;b) = \{b_1, \ldots b_k \}b$. Let us call operations $E_{1,k}$ left multibraces, and operations $E_{k,1}$ right multibraces. \\

In Section 5 of \cite{Ka} the author constructs (left) hGa structure on $B = \operatorname{Cobar}(A)$ for a bialgebra $A$. For tensors $x = x^{(1)} \otimes \ldots \otimes x^{(n)} \in B$ 
 and $y_1$, $y_2$, $\ldots$, $y_k$ $\in B$, the left multibrace $E_{1,k}$ is given by 
\begin{align*}
&   E_{1,k}(x;y_1,\ldots,y_k) = \\
&     \sum_{1 \leq i_1 < \ldots < i_k \leq n} \pm x^{(1)} \otimes \ldots  \otimes (\Delta^{|y_1|-1}(x^{(i_1)}) \cdot y_1) \otimes \ldots \otimes x^{(n)}.
\end{align*}

By $|y|$ we mean the length of tensor, and if $|x|=n<k$ then the multibrace vanishes. \\

One can similarly define (right) hGa structure on the same $B$. For tensors $x_1$, $x_2$, $\ldots$, $x_k$ $\in B$ 
 and $y = y^{{1}} \otimes \ldots \otimes y^{(n)} \in B$, the right multibrace $E_{k,1}$ is given by 
\begin{align*}
& E_{k,1}(x_1,\ldots,x_n;y) = \\
& \sum_{1 \leq i_1< \ldots < i_k \leq n} \pm y^{(1)} \otimes \ldots \otimes (x_1 \cdot \Delta^{|x_1|-1}(y^{(i_1)}))\otimes \ldots \otimes y^{(n)}.
\end{align*}

Now the formula \eqref{eqref:tensor} can be rewritten as 

$$(a \otimes b)_n = \sum_{i_1+ \ldots + i_k = n} \{ a_{i_1},\ldots,a_{i_k} \}b_k$$

and the formula \eqref{eqref:tensor2} can be rewritten as 

$$(a \otimes b)_n = \sum_{i_1+ \ldots + i_k = n} a_k \{ b_{i_1},\ldots,b_{i_k} \}$$

\begin{rem}
The results of \cite{ACD} on tensoring morphisms also work in our generality of non-commutative DG-Hopf algebra. However, extracted from its natural (operadic) framework, the formula looks totally unenlightening: 
\begin{align*}
& (f \otimes g)_n = \\
& = \sum_{\substack{i_1+\ldots+i_k=n \\ 1 \leq m \leq k}}g_0( a_{i_1} \otimes \ldots \otimes a_{i_{m-1}} \otimes f_{i_m} \otimes x_{i_{m+1}} \otimes \ldots \otimes x_{i_k})(\Delta^{i_1-1} \otimes \ldots \otimes \Delta^{i_k-1})b_k \\
& + \sum_{i_1+\ldots+i_k=n} f_0 (x_{i_1} \otimes \ldots \otimes x_{i_k})(\Delta^{i_1-1} \otimes \ldots \otimes \Delta^{i_k-1}) g_k \\
& + \sum_{\substack{i+j=n \\ i_1+ \ldots + i_k = i \\ j_1+ \ldots j_l = j \\ 1 \leq m \leq k}} (a_{i_1} \otimes \ldots \otimes a_{i_{m-1}} \otimes f_{i_m} \otimes x_{i_{m+1}} \otimes \ldots \otimes x_{i_k} \otimes x_{j_1} \otimes \ldots \otimes x_{j_l}) \\
& (\Delta^{i_1-1} \otimes \ldots \otimes \Delta^{i_k-1})b_k \otimes (\Delta^{j_1-1}\otimes \ldots \otimes \Delta^{j_l-1})g_l.
\end{align*}

We do not spell out the signs here, since the formula is already sufficiently intimidating in their absence. The tensor product of morphisms given by this formula is associative up to homotopy, and respects compositions up to homotopy. Packaging the data of all these higher homotopies is the goal of our ongoing project.
\end{rem}

\appendix
\section{Homotopy limit in DG-algebras}
\label{appen}
For any combinatorial model category $\C$ and a diagram $X$ of  the shape $\Delta$, one can use Bousfeld-Kan formula to find the homotopy limit as the {\em fat} totalization, see Example 6.4 in \cite{AO1}:
$$ \holim_{\Delta} X = \int_{\Delta^+} R(X^n)_n $$
where $R$ is some functor $\C \to \C^{\Delta^{opp}}$ which sends an object $c \in \C$ to its simplicial resolution, i.e. a Reedy-fibrant replacement of the constant simplicial diagram with value $c$. \\

We first present functorial simplicial resolutions for $\C \simeq \dgvect$, and then extend the construction to $\C \simeq \dgalg$. We then apply the fat totalization formula to compute the homotopy limit of a cosimplicial system associated with a DG-bialgebra.

\subsection{Simplicial resolutions in $\dgvect$}
Let us present functorial simplicial resolutions for $\dgvect$. \\

Recall a simplicial vector space $X_\bullet$ is under Dold-Kan correspondence sent to its Moore complex $N(X)^\bullet$, given by
$N(X)^{-n} = X_n/D_n$,
where $D_n$ is the degenerate part of $X_n$. The differential is the alternating sum of faces. \\

For $n \geq 0$, let $\mathsf{k}\Delta[n]$ be the linearization of standard simplex, and set $L^n = N(\mathsf{k}\Delta[n])$. Explicitly, this complex is spanned by elements $f_{i_0 <  \ldots < i_k}$ of degree $-k$ for $k \geq 0$, with $i_0 \geq 0$ and $i_k \leq n$ -- these are the nondegenerate simplices of $\Delta[n]$ that correspond to faces with vertices $i_0, \ldots, i_k$. The differential in this basis is 
$$d(f_{i_0 <  \ldots < i_k}) = \sum_{j=0}^{k}(-1)^jf_{i_0<\ldots<\widehat{i_{j}}<\ldots<i_k}$$
where $\widehat{i_j}$ denotes dropping this index. Due to functoriality of $N$, $L^\bullet$ is a cosimplicial system of complexes. For a map $\phi \co [n] \to [m]$ in $\Delta$, the corresponding map $\phi_*\co L^n \to L^m$ is given by 
$$\phi_*(f_{i_0 <  \ldots < i_k}) = \begin{cases} f_{\phi(i_0)< \ldots < \phi (i_k)} & \textrm{if }\phi|_{\{i_0,\ldots,i_k\}} \textrm{ is injective} \\ 0 & \textrm{otherwise} \end{cases}$$

\begin{prop}
\label{vectpower}
For $X \in \dgvect$, the simplicial system $X^{[-]}$ gives a simplicial resolution of $X$, i.e. it is Reedy-fibrant, and there exists a map $\operatorname{const}(X) \to X^{[-]}$ that is a levelwise quasiisomorphism.
\end{prop}

\begin{proof}
The map $r \co X \to X^{[n]}$ is is given by $x \mapsto r(x)$ where $r(x)(f_i)=x$ for all $i$, and $r(x)(f_{i_0<\ldots<i_k})=0$ when $k>0$. This respects differentials: we have 
$$r(d_X(x))(f_i) = d_X(x) = d_X(r(x)(f_i)) - r(x)(d_{L^{n}} (f_i)) = d_{X^{[n]}}(r(x))(f_i) $$
and 
$$(d_X(x))(f_{i<j}) = 0 = d_X(0)-r(x)(f_i-f_j)= d_{X^{[n]}}(r(x))(f_{i-j})$$
and for $k>1$
$$r(d_X(x))(f_{i_0 < \ldots <i_k}) = 0 = d_{X^{[n]}}(r(x))(f_{i_0 < \ldots <i_k})$$
because in $d(f_{i_0 < \ldots <i_k})$ all summands have degree strictly less than 0, so $r(x)$ vanishes on them. \\

We check that $r$ is a quasiisomorphism. We first check that it is injective on cohomology. Let $x \in X$ be a closed element such that its image vanishes in cohomology, $r(x) = d_{X^{[n]}}(s)$ for some $s \co L^n \to X$. Then 
$$x = r(x)(f_0) = d_{X^{[n]}}(s)(f_0) = d_{X}(s(f_0))-s(d_{L}(f_0))$$
so $x = d_{X}(s(f_0))$, i.e. it vanishes in cohomology. \\

We now check $r$ is surjective on cohomology. Let $s: L^n \to X$ be a closed morphism. Then $r(s(f_0))-s = d_{X^{[n]}}(t)$, where
$$t(f_0) = 0$$
$$t(f_{i}) = s(f_{0<i}) \textrm{ if }i>0$$
and in general,
$$t(f_{i_0<\ldots< i_k}) = \begin{cases} s(f_{0< i_0<\ldots< i_k}) & \textrm{if } i_0>0 \\
0 & \textrm{if } i_0=0 \end{cases} $$
For different $n$, these maps $r^{(n)}$ are consistent with cosimplicial structure: for $\phi \co [m] \to [n]$ we have  
$$r^{(m)}(x)(f_{i_0<\ldots<i_k}) = \begin{cases} x & k=0 \\ 0 & k>0 \end{cases}$$ and 
$$\phi^*(r^{(n)}(x))(f_{i_0<\ldots<i_k}) = r^{(n)}(x)(\phi_*(f_{i_0<\ldots<i_k})) = \begin{cases} x & k=0 \\ 0 & k>0 \end{cases}$$
We are left to verify Reedy fibrancy, i.e. that matching maps are fibrations in $\dgvect$, i.e. surjections. By definition, the $n^{\mathrm{th}}$ matching object $M_n$ is 
$$M_n = \lim_{\delta([n] \downarrow (\Delta^{\mathrm{op}})_{-})}X^{[-]} = \lim_{[m]\hookrightarrow [n]}X^{[m]}.$$
These are morphisms from a subcomplex of $\overline{L}^n \subset L^n$ that is spanned by everything except $f_{0<\ldots<n}$. The matching map $m^n: X^{[n]} \to M_n$ is given by forgetting the value of a morphism $L^n \to X$ on $f_{0<\ldots<n}$. This is a surjection of chain complexes, as any morphism $\overline{L}^n \to X$ can be extended to a morphism $L^n \to X$ by assigning any value to $f_{0<\ldots<n}$.
\end{proof}

\subsection{Simplicial resolutions in $\dgalg$}
We now enhance our construction of simplicial resolutions from $\dgvect$ to $\dgalg$. The result is motivated by Holstein resolutions in $\dgcat$ (see \cite{Hol}, \cite{AP}) but simpler. 
%Our method consists in upgrading $L^\bullet$ to a cosimplicial DG-coalgebra...

\begin{prop}
The cosimplicial system of complexes $L^\bullet$ can be upgraded to a cosimplicial system of DG-coalgebras, by introducing the following comultiplication:
$$\Delta(f_{i_0 <  \ldots < i_k}) =  \sum_{j=0}^{k} f_{i_0 < \ldots < i_j} \otimes f_{i_j < \ldots < i_k}$$
\end{prop}

\begin{proof}
Compatibility with differentials and and with cosimplicial structure is checked by an elementary explicit computation.
\end{proof}

\begin{rem} Conceptually this is the comultiplication in standard simplices that is responsible for the existence of cup-product in singular cohomology. \end{rem}

Now, for any monoidal DG-category $\C$, if $X$ is a coalgebra in $\C$ and $Y$ is an algebra in $\C$, then the complex $\C(X,Y)$ is a DG-algebra by means of convolution:
$$\C(X,Y) \otimes \C(X,Y) \simeq \C(X \otimes X, Y \otimes Y) \xrightarrow{ (\Delta_X, \mu_Y)} \C(X,Y)$$

We are working in the case when $\C$ is the category of chain complexes, $\dgvect$. Coalgebras in $\dgvect$ are DG-coalgebras and algebras in $\dgvect$ are DG-algebras. So for $A$ a DG-algebra, the Hom-complex $\Hom^{\bullet}(L^n,A)$ has a DG-algebra structure. Denote this algebra by $A^{[n]}.$

\begin{prop}
For a DG-algebra $A$, the simplicial system $A^{[-]}$ gives a simplicial resolution of $A$, i.e. it is Reedy-fibrant, and there exists a map $\operatorname{const}(A) \to A^{[-]}$ that is a levelwise quasiisomorphism.
\end{prop}

\begin{proof}
The map $r \co A \to A^{[n]}$ is exactly the same as in the case of $\dgvect$ - namely, $a \mapsto r(a)$ where $r(a)(f_i)=a$ for all $i$, and $r(a)(f_{i_0<\ldots<i_k})=0$ when $k>0$. We check that this map is compatible with multiplication:
$$ (r(a)*r(b))(f_{i}) = \mu_A (r(a) \otimes r(b))(f_i \otimes f_i) = ab = r(ab)(f_i). $$
and for $k>0$
$$ (r(a)*r(b))(f_{i_0 < \ldots <i_k}) = 0 = r(ab)(f_{i_0 < \ldots <i_k}) $$
because in every summand of $\Delta(f_{i_0 < \ldots <i_k})$ at least one of the components has degree strictly less than 0. \\

It was already verified in the proof of Proposition \ref{vectpower} that $r$ is compatible with differentials and a quasiisomorphism.\\

In checking Reedy fibrancy we are left to notice that the subcomplex $\overline{L}^n \subset L^n$ (spanned by all basis elements except for $f_{0<\ldots<n}$) is actually a subcoalgebra, so matching objects and matching maps in $\dgalg$ are the same as in $\dgvect$.
\end{proof}

\subsection{Fat totalizations in $\dgvect$ and $\dgalg$}
Let $X^\bullet$ be the cosimplicial complex in whose homotopy limit we are interested. Then 
$$ \holim_{\Delta} X^\bullet = \int_{\Delta^+} (X^n)^{[n]} = \operatorname{Eq}\left(\prod_{n\geq0}\Hom^\bullet(L^n,X^n) \rightrightarrows \prod_{[m] \hookrightarrow [n]}\Hom^\bullet(L^m,X^n)\right).$$
This is the complex $\operatorname{Nat}_{\Delta^+} (L^\bullet,X^\bullet)$ of natural transformations between two functors $\Delta^+ \to \dgvect$.

\begin{prop}
\label{holim} As a graded vector space, the homotopy limit of a cosimplicial vector space $X^\bullet$ is given by
$$\operatorname{holim}_\Delta X^\bullet = \prod_{n=0}^{\infty} X^n[-n]. $$
For an element $x = (x_0,x_1,\ldots)$, its differential is given by
\begin{equation}
\label{eq:difvect}
    d(x)_n = d_{X^n}(x_n)-\sum_{i=0}^n \partial^{(0\ldots\widehat{i}\ldots n)}(x_{n-1}).
\end{equation}

\end{prop}
\begin{proof}
A natural transformation $\phi\co L^\bullet \to X^\bullet$ consists of maps $\phi^n: L^n \to X^n$ for all $n$. For all indexing subsets $I$ smaller than $ \{ 0<\ldots<n \}$, the generator $f_I$ is in the image of $i^*\co L^m \to L^n$ for some $i\co [m] \hookrightarrow [n] \in \Delta^+$, $m<n$. Thus the only part of $\phi^n$ that is not determined by $\phi^m$ for $m<n$ is its value $\phi^n(f_{0<\ldots<n})$. So the graded isomorphism 
$$\operatorname{Nat}_{\Delta^+} \xrightarrow{\simeq} \prod_{n=0}^{\infty} X^n[-n]$$
is given by $\phi \mapsto \phi^0(f_0) \times \phi^1(f_{0<1}) \times \phi^2(f_{0<1<2})\ldots = ( \phi^n(f_{0<\ldots<n}) )_{n=0}^\infty$.\\

The differential comes from the differential in $\prod_{n\geq0}\Hom^\bullet(L^n,X^n)$. Let $x=(x_0,x_1,\ldots)$ be an element with the corresponding natural transformation $\phi=(\phi^0,\phi^1,\ldots)$ with $\phi^n(f_{0<\ldots<n})=x_n$. Then we have 
\begin{align*}
& d_{\Hom}(\phi^n)(f_{0<\ldots<n}) =  d_{X^n}(\phi^n(f_{0<\ldots<n}))-\phi^{n}(d_{L^n}(f_{0<\ldots<n})) \\
& = d_{X^n}(x_n) - \sum_{i=0}^n \partial^{(0\ldots\widehat{i}\ldots n)}(x_{n-1})
\end{align*}
\end{proof}
Now let $A^\bullet$ be the cosimplicial DG-algebra in whose homotopy limit we are interested.

\begin{prop}
\label{holimalg}
The underlying complex of $\holim_{\Delta}(A^\bullet)$ is as described in Proposition \ref{holim}. For two elements $a = (a_0,a_1,\ldots)$ and $b = (b_0,b_1,\ldots)$, their product is given by
\begin{equation}
\label{eq:holimmult}
    (a\cdot b)_n = \sum_{i=0}^n \partial^{(0\ldots i)}(a_i) \cdot \partial^{(i\ldots n)}(b_{n-i}) 
\end{equation} 
\end{prop}
\begin{proof}
The description of the underlying complex follows from the fact that simplicial resolutions in $\dgvect$ are the underlying complexes of simplicial resolutions in $\dgalg$. We now recover the multiplication given by convolution. Let $\phi$ and $\psi$ be two natural transformations corresponding to $a = (a_0,a_1,\ldots)$ and $b = (b_0,b_1,\ldots)$. Then 
\begin{align*}
& (\phi * \psi)^n (f_{0<\ldots<n}) = (\phi^n * \psi^n) (f_{0<\ldots<n}) = \mu_{A^n} (\phi^n \otimes \psi^n)\Delta_{L^n}(f_{0<\ldots<n}) \\
& = \mu_{A^n} (\phi^n \otimes \psi^n) \left( \sum_{i=0}^n f_{0<\ldots <i} \otimes f_{i<\ldots<n} \right) = \sum_{i=0}^n \phi^n(f_{0<\ldots <i}) \cdots \psi^n(f_{i<\ldots<n}) \\
& = \sum_{i=0}^n \partial^{(0\ldots i)}(\phi^i(f_{0<\ldots<i})) \cdot \partial^{(i\ldots n)}(\psi^{n-i}(f_{0<\ldots<n-i})) \\
& = \sum_{i=0}^n \partial^{(0\ldots i)}(a_i) \cdot \partial^{(i\ldots n)}(b_{n-i})
\end{align*}

\end{proof}

\subsection{Application to the cosimplicial system of a DG-bialgebra}
Let $A$ be a DG-bialgebra, and let $A^\bullet$ be its associated cosimplicial system of DG-algebras, as in \eqref{eq:cosys}. Let us use the above formulas to compute its homotopy limit.

\begin{prop}
$\holim_{\Delta} (A^\bullet) \simeq \operatorname{Cobar}_{\operatorname{coaug}}(A).$
\end{prop}
\begin{proof}
By Proposition \ref{holim}, the underlying graded vector space of the homotopy limit is $\prod_{i=0}^n A^{\otimes i}$, which is exactly the underlying graded vector space of $\operatorname{Cobar}_{\operatorname{coaug}}(A)$. With the data of appropriate faces, the formula \eqref{eq:difvect} translates into the differential of the reduced Cobar construction, and the formula \eqref{eq:holimmult} translates into tensor multiplication.
%$$d = \sum_{}^{}\id^{\otimes} \otimes d_{A} \otimes \id^{\otimes} + 1 \otimes \id^{\otimes n} + \sum_{i=0}^n \id \otimes \Delta_{A} \otimes \id + \id^{\otimes n} \otimes 1$$.
%This is precisely the differential of the reduced Cobar construction
\end{proof}

\end{document}